\documentclass[12pt]{article}
\usepackage[a4paper]{geometry}
\usepackage{amsthm}
\usepackage{amssymb}
\usepackage{amsfonts}
\usepackage{MnSymbol}
\usepackage{amsmath}
\usepackage{eufrak}

\newtheorem{theorem}{Theorem}

\newtheorem{corollary}[theorem]{Corollary}
\newtheorem{lemma}[theorem]{Lemma}

\newtheorem{definition}[theorem]{Definition}

\newtheorem{proposition}[theorem]{Proposition}

\def\K{{\mathbb K}}
\def\k{{\mathbb K}}
\def\F{{\mathcal F}}

\def\kx{\K\langle\!\langle x,y \rangle\!\rangle}
\def\ww{{\rcirclearrowleft}}
\def\ar{{\nearrow}}
\def\deg{{\rm deg}\,}
\def\dim{{\rm dim}\,}
\def\spann{\hbox{ span}\,}

\begin{document}
\def\F{{\mathbb F}}
\title{Classification of contraction algebras and pre-Lie algebras associated to  braces and trusses}
%  Pre-Lie algebras associated to  braces and trusses, and classiffication results on contraction algebras}
\author{Natalia Iyudu}
%, Agata Smoktunowicz}
\date{}
\maketitle

\bigskip
\centerline{\it Dedicated to the memory of Ernest Borisovich Vinberg}
\bigskip

\begin{abstract}

We develop tools for classification of contraction algebras and apply these to solve the problem on classification up to isomorphism  of 8 and 9 dimensional algebras corresponding to 3-fold flops. We prove that  there is only one up to isomorphism contraction algebra of dimension 8, and two algebras of dimension 9.
The formulae for the dimension of algebra, depending on the type of the potential are obtained.

 In the second part of the paper we show that associated graded structure to brace and truss with appropriate descending ideal filtration is pre-Lie.

\end{abstract}


\section{Introduction}

We consider here
%two  questions, in both of which working in the realm of infinite power series is essential.
first  classification of 8 and 9 dimensional two generated contraction algebras.
Contraction algebras  were introduced in  Donovan-Wemyss  work on minimal model program and noncommutative resolution of singularities \cite{W, WD}. Namely,
 they serve as noncommutative invariants attached to a birational flopping contraction:
$f: X \to Y$ which  contracts rational curve $C \simeq {\Pr}^1 \subset  X$ to a point, where  $X$ is a smooth quasi-projective 3-fold.
It is a new, essentially noncommutative  invariant of curve contraction, which is not a number, as ones previously known, such  as Gopakumar-Vafa invariants, but instead an associative noncommutative algebra.  It recovers all known invariants in a natural way, and  as it is argued  in \cite{BW}, the contraction algebras are finer invariants of 3-fold flops, than various curve-counting theories, as there are examples where curve-counting invariants are the same, but contraction algebras are not isomorphic. We will see these examples here within the classification of 9 dimensional contraction algebras.
 In \cite{BW},  the question on this classification was raised and it was noted that 'the isomorphism problem is delicate and is in general also difficult'.
Thus we develop here techniques which can serve to solve this problem of classification up to isomorphism. Application of this technique allows us to classify all 8 and 9 dimensional  contraction algebras on 2 generators (corresponding to 3-fold flops). We hope that these classification  results  can help to better understand the ways contraction algebras capture geometric information. In particular, where the refinements, for example in shape of derived contraction algebras \cite{Matt} are needed.

It is known due to result of Van den Bergh \cite{MB} that contraction algebras are potential, that is defining relations are noncommutative partial derivatives of a potential.
In view of the conjecture of Wemyss and Donovan \cite{WD}, (see also \cite{Davison} for evidences for the conjecture), saying that any potential algebra  (of rose quiver, and more generally of a quiver from a certain wider list) can be realised as a contraction algebra, we approach the study of contraction algebras via the study of potential algebras by methods similar to the ones from \cite{ISP, ISCl, Skl1, Skl2}, for example.

In \cite{AN} we answered  question due to M.Wemyss on the minimal dimension for a contraction algebra in two generators, and also found  a conditions on the  potential necessary to produce a finite dimensional  algebra (or an algebra of linear growth)\cite{AN, ISPA}.
It turned out that
the minimal dimension possible for a contraction algebra is 8, and the potential should be nonhomogeneous and contain terms of degree 3, in order algebra to be finite dimensional or of linear growth.

Here we prove the following classification results.

\begin{theorem} There is one up to isomorphism contraction algebra $A=\kx/id(P)$ of dimension 8, and it is defined by the potential  $x^3+y^3+xyxy\ww$. There are two up to isomorphism contraction algebras $A=\kx/id(P)$  of dimension 9, and they are defined by potentials  $x^2y\ww + y^4$ and $x^2y+y^4+y^5$.
\end{theorem}
%'deciding when two finite-dimensional algebras are not isomorphic is tricky'.
%, especially in the situation here, when by design all s

We start the proof by showing that invertible linear changes of variables can bring a cubic potential to one of the three forms: with the cubic part $x^2y\ww$, or $x^3$ or $x^3+y^3$. We show that 8-dimensional algebras can appear only from the potential with cubic term $x^2y\ww$,
and 9-dimensional algebras  only from the potential with cubic term $x^3+y^3$.
We also found the formulae for the dimension of the contraction algebra. In case of the potential with cubic term $x^2y\ww$ combination of Lemma~\ref{zamx2y}, Proposition~\ref{gen} and Corollary~\ref{1} provides the following dimension counting result.

\begin{theorem} Let $A=\kx /id(F)$, be a finite dimensional algebra with the potential $F=x^2y\ww+F_4+...+F_r$, then an invertible changes of variables can bring the potential $F$ to $x^2y+y^4p(y)$, and the dimension of $A$ is:
$\dim A =3(2n+3)$,  if $k=2n,$ $ \dim A= 4n+k+9$, if $k<2n$, ($k$ is odd). Here $k$ is degree of polynomial $p(y)$, and $2n$ is degree of its even part.
\end{theorem}

The second part of the paper is dedicated to the connections between pre-Lie algebra structures with braces and trusses. There were noticed a numerous connections between braces and pre-Lie algebras (\cite{AG, Rump, Ag}) inspired by the connections between Lie algebras and groups. We focus here on  the fact that whenever the brace is endowed with decreasing ideal filtration with zero intersection, the associated graded structure is a pre-Lie algebra.

 {\it Pre-Lie structures} are prominent by their numerous appearances in different areas.
%In \cite{Rump}, Wolfgang Rump showed that there exists a correspondence  between left nilpotent right $\mathbb R$-braces and pre-Lie algebras.  %Pre-Lie algebras
They were introduced by Vinberg under the name {\it left-symmetric algebras} {(LSA)}  in  his study of convex homogeneous cones \cite{Vin}, and a graded version of pre-Lie algebras  appeared at the same time in the work of Gerstenhaber \cite{Gerst} as a structure on the Hochschild complex of an associative algebra. In fact this structure is present in rooted trees with drafting operation and can be traced back to the work of A.Cayley \cite{Cayley}.

{\bf Definition}. Pre-Lie algebra is a structure $A=(A, \cdot)$ with one binary (nonassociative) operation $A\cdot A \to A$, satisfying the left symmetricity
of the associator: $(a,b,c)=(b,a,c)$, where $(a,b,c)=(a\cdot b) \cdot c - a\cdot (b \cdot c)$.

For more information on this structure see for example \cite{Burde, M, MP}.

{\it Braces} were introduced by Rump \cite{rump} to describe all involutive set-theoretic solutions of the Yang-Baxter equation.
%This approach subsequently found applications in several other research areas.

 \begin{definition} A set $(A, *, +) $ with two binary operations is called a brace, if $(A,+)$ is an abelian group,  $(A,\circ)$ is a group and
the following mix of associativity and distributivity axioms holds:
%$a *b =A\circ b -a-b$, from the axiom of associativity in $(A,\circ)$ we get the same axiom as in brace:
$$(a*b+a+b)*c=a*c+b*c+a*(b*c).$$
Here we denote $a\circ b =a+b+a*b.$
\end{definition}
%There  is a close connection between braces and  groups, namely the with operation...

 We consider also the notion of {\it truss}, which generalise braces. It was introduced by T.Brzezinski \cite{Brz} to incorporate the notion of associative ring and the brace.

 \begin{definition} A set $(A, \circ, +) $ with two binary operations is called a truss, if $(A,+)$ is an abelian group,  $(A,\circ)$ is a semigroup and
$$a\circ(b+c)+a=a\circ b+a\circ c +\alpha(a)$$
where $\alpha$ is some function $\alpha: A \to A.$
\end{definition}

In this second part of the paper we prove the following.

\begin{theorem}\label{br}
% Let $\mathbb F$ be a field which is either a finite field of $p$ elements with $p$-prime, or the field of rational numbers.
 Let $A$ be an $\k$-brace,   endowed with a descending ideal filtration  with zero intersection, then
 the associated graded structure of a brace  $A$ is a pre-Lie algebra over $\mathbb F$.

\end{theorem}

 %To prove this theorem w
 We notice also that the result of lemma 15 \cite{Engel} extends from the nilpotent case to the case of arbitrary descending ideal filtration with zero intersection, and spell out the right distributivity formula, which holds in a completion $\widehat A$ of $A$ in the topology defined by the filtration (see Lemma~\ref{di} in Section~\ref{br}).

%\begin{lemma}
%Let $B$ be a left brace endowed with the descending filtration  $B=\bigcup B_i$ where $B_i$ are ideals in $B$, $ %B_{i+1} \subset B_{i}, \, B_i * B_j \subset B_{i+j}$, satisfying  $\bigcap\limits_{i=1}^{\infty} B_i =0$. Then %the right distributivity formula holds in $\widehat B$, for any $a,b,c \in B$:
%$$(a+b)*c=a*c+b*c+ \sum\limits_{i=1}^{\infty}(-1)^{i+1} (d_i*d_i')*c -d_i*(d_i'*c).$$
% here $d_0=a, d_0'=b$ and for $i \geq 1$ $d_i$th are defined as $d_i=d_{i-1}+d_{i-1}'.$
%\end{lemma}

It turns out that the
analogous results to Theorem~\ref{br} holds for trusses under a mild condition on filtration (Corollary~\ref{Ctr}, Section~\ref{tr}).

\begin{corollary}\label{Ctr} Let $B$ be a truss endowed with
 descending filtration $B=\bigcup B_i$ where $B_i \triangleleft  B$, $ B_{i+1} \subset B_{i}, \, B_i * B_j \subset B_{i+j}$, satisfying  $\bigcap\limits_{i=1}^{\infty} B_i =0$,  and such that $\deg \alpha(a) > 2$. Then the associated graded structure
$B_{gr}= \oplus B_i/B_{i+1} =\oplus \bar B_i$  is a pre-Lie algebra.
\end{corollary}

\section{Asknowledgements}
I would like to thank Michael Wemyss for enlightening  conversations and for asking very interesting questions on contraction algebras, I am also very thankful to  Agata Smoktunowicz for introducing me to braces and  communicating a series of inspiring questions.
I am grateful to MPIM and IHES for support and excellent research atmosphere.
% on contraction algebras, as well as on braces and trusses. ...

\section{Preliminary facts}

Let $\K\langle x,y\rangle$ be a free noncommutative algebra on free generators $x,y$, and $\kx$ be a free algebra of formal power series on $x,y$.
Let A be the quotient of the formal power series in two variables by an ideal generated by relations $R$:
$A=\kx/ id_{\kx} (R)$, and $B=\k\langle x,y \rangle /(R)$, be the quotient of the free associative algebra by the ideal generated by the same relations, but in $\k\langle x,y \rangle$.

Let us remind the notion of {\it completion} of an algebra $B$.
%, and see that it is isomorphic to the algebra $A$.
We assign to variables $x$ and $y$ degree 1, and say that polynomial $p \in \K\langle x,y\rangle$ (a series $p \in \kx$) have a {\it degree} $n$, if the {\it minimal} degree of monomial on $x,y$, present in $p$ (with nonzero coefficient) is $n$. We will use this definition of degree throughout the paper.

\begin{definition} We say that the ideal $\overline I$ is the completion of the ideal $I \triangleleft \K\langle x,y\rangle$ if
$\overline I= \bigcap\limits_{n=1}^{\infty} I^{[n]}$,
where $I^{[n]} =I + id\ \text{(monomials of}\ \deg (n+1)).$ Obviously, $I^{[n+1]} \subseteq I^{[n]}$, and $I \subset \overline I$.
The algebra $\overline A = \K\langle x,y\rangle /\overline I$  is then called a completion algebra of $A=\K\langle x,y\rangle /I$.
\end{definition}

Note, that the completion algebra $\overline A$ is a quotient of algebra $A$ itself.

It is easy to see from this definition that $\kx /id_{\kx} (R) = {\K\langle x,y\rangle/ \overline {id_{\K\langle x,y\rangle} (R)}}$.

We will mention in this section few more facts on this construction, which we  use freely throughout the text.

Now let us give two equivalent definition of an algebra given by a potential.

\begin{definition}Let $F$ be a cyclic polynomial $F\in \K\langle x,y \rangle /[\K\langle x,y \rangle, \K\langle x,y \rangle] ,$
the {\it potential algebra} $A(F)$,
%given by two relations, which are partial derivatives of $F$, i.e. $A(F)$
is the
factor of ${\mathbb K}\langle x,y\rangle$ by the ideal $I_F$ generated by
$\frac{\partial F}{\partial x}$ and $\frac{\partial F}{\partial y}$, where the linear maps
$\frac{\partial}{\partial x}, \frac{\partial}{\partial y}:{\mathbb K}\langle x,y\rangle\to {\mathbb K}\langle x,y\rangle$
%$\frac{\partial}{\partial y}:{\mathbb K}\langle x,y\rangle\to {\mathbb K}\langle x,y\rangle$
are defined on
monomials as follows:

$$
\frac{\partial w}{\partial x}=\left\{\begin{array}{ll}u&\text{if $w=xu$,}\\ 0&\text{otherwise,}\end{array}\right.
\qquad \frac{\partial w}{\partial y}=\left\{\begin{array}{ll}u&\text{if $w=yu$,}\\ 0&\text{otherwise.}\end{array}\right.
$$

\end{definition}

\begin{definition}\label{Ginzburg} (Ginzburg )\cite{Ginzburg}
Associated to any (not necessary cyclic) polynomial $\Phi \in \K\langle x,y\rangle $ the potential algebra is the
factor of ${\mathbb K}\langle x,y\rangle$ by the ideal $I_{\Phi}$ generated by
$\frac{\partial \Phi}{\partial x}$ and $\frac{\partial \Phi}{\partial y}$, where the linear maps
$\frac{\partial}{\partial x}, \frac{\partial}{\partial y}:{\mathbb K}\langle x,y\rangle\to {\mathbb K}\langle x,y\rangle$
%$\frac{\partial}{\partial y}:{\mathbb K}\langle x,y\rangle\to {\mathbb K}\langle x,y\rangle$
are defined on
monomials as follows.
 Given a monomial $u = z_{i_1}z_{i_2} ...z_{i_r}$, $z_i=x$ or $y$,
define  $\frac{\partial u}{ \partial z_j} = \sum\limits_{s:i_s=j} z_{i_{s+1}} z_{i_{s+2}}...z_{i_r} z_{i_1} ...z_{i_{s-1}},$
 then extend this deﬁnition by linearity.
 \end{definition}

 The difference between these two definitions is that the first one works only for cyclic polynomials, and the second can be defined for any polynomials. However the class of potential algebras they both produce is the same.

 Note also one simple fact about the syzygy, which holds for any algebra given by a cyclic potential.

 \begin{lemma}\label{syz}
For every $F\in\K\langle x,y\rangle$ such that $F_0=0$, $F=x\frac{\partial F}{\partial x}+y\frac{\partial F}{\partial y}$. Furthermore, the equality $F=\frac{\partial F}{\partial x}x+\frac{\partial F}{\partial y}y$ holds if and only if $F$ is cyclicly invariant. In particular, $[x,\frac{\partial F}{\partial x}]+[y,\frac{\partial F}{\partial y}]=0$ if and only if $F$ is cyclicly invariant.
\end{lemma}

\begin{proof}
Trivial.
\end{proof}

\begin{proposition}\label{prop1} Let $A=\kx / id_{\kx} (F)$ where $F$ is not necessarily homogeneous  polynomial, is a finite dimensional algebra. Then A is nilpotent.
\end{proposition}

\begin{proof} Since $A$ is finite dimensional, if we take enough  powers of $a$, they will be linearly dependant, so $a$ is algebraic:
$p(a)=\alpha_0+\alpha_1 a +\alpha_2 a^2+...+\alpha_n a^n=0.$ Note that if  $\alpha_0$ would be nonzero, then the polynomial is invertible in $\kx$ and it can not be equal to zero. Let us take the maximal power of $a$ such that $p(a)=a^k (1+\alpha'_1 a +\alpha'_2 a^2+...)$ since the second multiple is invertible in $\kx$, we have $a^k=0$. Thus $A$ is nil, and being finite dimensional it is nilpotent.
\end{proof}

\begin{proposition}\label{prop2} The change of variables (not necessary linear) in the potential coincide with the same change of variables in the relations.
\end{proposition}

\begin{proof} The proof is based on the following lemma.

\begin{lemma} The formula for the derivative of the composition works, and easy to check in the Ginzburg definition of the potential algebra.
 Let $G=F(u(x,y), v(x,y))$, where $u(x,y), v(x,y)$ are monomials, then

$$\frac{\partial G}{\partial x} = \partial_1 F (u(x,y), v(x,y)) \cdot  \frac{\partial u(x,y)}{\partial x}
 +
\partial_2 F (u(x,y), v(x,y)) \cdot  \frac{\partial v(x,y)}{\partial x} $$

$$
\frac{\partial G}{\partial y} = \partial_1 F (u(x,y), v(x,y)) \cdot  \frac{\partial u(x,y)}{\partial y} +
\partial_2 F (u(x,y), v(x,y)) \cdot  \frac{\partial v(x,y)}{\partial y}.$$

Here $\cdot$ stands for 'noncommutative multiplication' of monomials.
\end{lemma}

\begin{proof}Easy check.
\end{proof}

 In our case, $G=F(x+m(x,y), y)$, where $m(x,y)$ is a monomial on $x,y$ of degree two or bigger, then

$$
\frac{\partial G}{\partial x} = \frac{\partial F}{\partial x} (x+m(x,y), y) \cdot (1+ \frac{\partial m(x,y)}{\partial x})
$$

$$
\frac{\partial G}{\partial y} = \frac{\partial F}{\partial y} (x+m(x,y), y)(1+ \frac{\partial m(x,y)}{\partial y})+
\frac{\partial F}{\partial y} (x+m(x,y), y)
%  + \cdot (1+ \frac{\partial m(x,y)}{\partial y}) +
.$$

Since $1+ m(x,y)$ is invertible in $\kx$, we see that ideals generated by $\frac{\partial G}{\partial x}, \frac{\partial G}{\partial y}$ and $\frac{\partial F}{\partial x} (x+m(x,y), y), \frac{\partial F}{\partial y} (x+m(x,y), y)$ do coincide.
% Hence we use this definition, and see ...?
\end{proof}

\begin{proposition}\label{prop3} Let $A=\kx/ id_{\kx} (\partial_x F, \partial_y F)$. If this algebra is nilpotent, then the change of variables  of the form
$x \to x+ f(x,y); y \to y$,
where $f(x,y) $ is a polynomial on $x,y$ of degree two or bigger, is invertible.
\end{proposition}

\begin{proof} The change of variables is given by triangular matrix with 1th on the diagonal, and this matrix is finite due to nilpotency condition.
%triangular matrix
\end{proof}

\begin{proposition}\label{radical} Let $A=\kx/ id_{\kx} (R)$, and $B=\k\langle x,y \rangle /(R)$. If $A$ is finite dimensional and nilpotent, then $A$ is a Jacobson radical of $B$: $A=Jac (B)$
\end{proposition}

%\begin{proof} ?
%\end{proof}

Another essential tool we will use here is the Gr\"obner bases theory, so we recall some basic terminology here.
For more detailed account see, for example, \cite{Mora}

Suppose in ${\mathbb K}\langle Y \rangle$ we have fixed some well-ordering (that is, ordering compatible with multiplication), for example, (left-to-right) degree-lexicographical ordering: we fix an order on variables, say $y_1 <...<y_n$, and compare monomials on $Y$ of the same degree lexicographically (from left to right). Then we say, that monomials of higher (or lower) degree are bigger in the ordering.  Polynomials are compared by their highest terms in this ordering. This order should have d.c.c.

\begin{definition}
Monomials $ u,v \in {\mathbb K}\langle Y \rangle$ form an ambiguity $(u,v),$ if for some $w\in {\mathbb K}\langle Y\rangle$, $uw=wv$.
\end{definition}

\begin{definition}
 Let $u,v$ be  two monomials $u,v$, which are highest terms of the elements $U,V$ from the ideal $I\in {\mathbb K}\langle Y \rangle: U=u+\tilde u, V=v+\tilde v$, where $\tilde u, \tilde v \in {\mathbb K}\langle Y \rangle$, smaller than $u,v\in \langle Y \rangle $ respectively: $\tilde u < u, \tilde v < v$. Then the resolution of the ambiguity $(u,v)$ formed by monomials $u,v$ is a polynomial $Uw-wV=\tilde u w-w\tilde v$, which is reducible to zero modulo  generators of an ideal.
\end{definition}

\begin{definition} A reduction on $\K \langle Y \rangle$ modulo generators of an ideal $f_i=\bar f_i+\tilde  f_i$, where $\bar f_i$ is a highest term of $f_i$, is a collection of linear maps defined on monomials as follows:
$r_{u \bar f_i v}(w)= u \tilde f_i v$, if $w=u \bar f_i v$, and $w$ otherwise.
\end{definition}

The polynomial is called reducible to zero if there exists a sequence of reductions modulo generators of an ideal, which results in zero.

In our arguments we often follow the Buchberger algorithm for construction of noncommutative Gr\"obner basis.

\section{Classification of 8-dimensional  contraction algebras}

Let $A$ be a potential completion algebra

$$A=\kx/ id_{\kx} (\partial_x F, \partial_y F)$$
 for some potential $F \in \k\langle x,y \rangle$.

 We are interested in the case, when potential  algebra is finite dimensional.
 As we have proved in \cite{AN, ISPA} it is only possible if the potential is nonhomogeneous and has terms of degree 3.
 General term of degree three is $F_3=a_1 x^3+a_2 x^2y\ww +a_3 xy^2\ww  +a_4 y^3$.
 Let us start with the

 \begin{lemma} By a non-degenerate linear transformation a general potential $F=F_3+f_4+...+F_r, F_3=a_1 x^3+a_2 x^2y\ww +a_3 xy^2\ww  +a_4 y^3$
 can be made into one of the four potentials:
 $F=0+F_4+...+F_r, F=x^3+F_4+...+F_r,  F=x^2y\ww +F_4+...+F_r, or  F=x^3+y^3+F_4+...+F_r.$
 \end{lemma}

\begin{proof} Indeed, consider the abelianisation of $F_3$:
$F_3^{ab}= a_1 x^3+3a_2 x^2y +3a_3 xy^2  +a_4 y^3.$ It can be decomposed as $F_3^{ab}=u_1(x,y)u_2(x,y)u_3(x,y)$ with $\deg u_i=1$.
If $F_3^{ab}\neq 0$, there are three possibilities: all lines $u_1,u_2,u_3$ are parallel: $u_1=\alpha u, u_2=\beta u, u_3=\gamma u, u=ax+by, (a,b) \neq(0,0)$, then there is a change of variables making the $F_3$ into $x^3$. Another possibility is that two of $u_1$ and $u_2$ are parallel, and they are not parallel to $u_3$, then $F$ can be made into $x^2y\ww$. If all three $u_1, u_2, u_3$ are pairwise nonparallel, then the linear transformation $x\to x', y\to y'$, where $u_1=x'+y', u_2=x'+\theta y', \theta^3=1$ will bring $u_1u_2u_3=(x+y)(x+\theta y)(x+\theta^2 y)$ to $x^3+y^3$.
\end{proof}

Thus, we will use the fact  that up to isomorphism  a nonzero degree 3 homogeneous cyclic  potential can be only either $x^3$ or $x^3+y^3$ or $xy^2\ww=xy^2+y^2x+yxy$.

\subsection{Potential with cubic term $x^2y\ww$}\label{x2y}

We will show first that in the cases of potential with $F_3=x^2y\ww$ and $F_3=x^3$ the dimension of algebra is bigger or equal than 9. Then main our consideration for the dimension 8 will go to the case of potential with $F_3=x^3+y^3$.

\begin{lemma}Let $A=\kx/ id_{\kx} (\partial_x F, \partial y F)$
 for the potential $F \in \k\langle x,y \rangle,
 F=x^2y\ww +F_4+...+F_r.$ Then $\dim A \geq 9$.
\end{lemma}

\begin{proof} Denote  components of the filtration on $A$ by $A_n$: $A=\bigoplus\limits_{n=0}^{\infty} A_n$. Here $A_n$ stands for polynomials (series) with the (lower) degree $n$, $A_0=\k$. We can consider associated graded algebra $A_{gr}=\prod\limits_{n=0}^{\infty}A_n/A_{n+1}$ where the $n$th graded component consists of series starting in degree exactly $n$. The dimensions of $A_{gr}$ and $A$ do coincide and are equal to $\sum\limits_{n=0}^{\infty} (A_n/A_{n+1})$.
%(since $A$ is finite dimensional).
 Obviously $A_1/A_2={\rm span} \langle x,y,\rangle_{\k}$, since the potential is of degree 3 and there are no relations  on monomials of degree smaller than two. The relations  in degree two are $\partial_y F_3= x^2+...$ and $ \partial_x F_3= yx+xy+...  $. Thus, modulo this ideal any occurrence of $x^2$ can be reduced to sum of monomials of higher degree, and $xy$ can be changed to $-yx$  plus sum of monomials of higher degree.
%\end{proof}

This means that the linear basis in $A_2/A_3$ consists of $x^2, yx$, so $\deg A_2/A_3=2 $

Let us call {\it n-normal monomials} those monomials  of degree $n$, which does not contain as submonomials elements of linear basis of $A_{n-1}/A_n$.

Consider now 3-normal monomials, that is those monomials of degree 3, which contain as a submonomials only already reduced monomials of degree 2: $x^2$ and $xy$. There are two of them: $x^3$ and $yx^2$. Thus the dimension of $A_3/A_4$ can be 2 or smaller, if there are relations on these monomials coming from  the ideal generated by the potential. But in this case there are no such relations, because the relations  on the level of degree 2, that is relations $y^2=0, xy=-yx$ with the ordering $x>y$ form a Gr\"obner basis. Thus $\dim A_3/A_4$=2.

Consider now $A_4/A_5$, 4-normal monomials are $x^4, yx^3$. Look how we can get dependances between them out of relations. We would multiply $\partial_x F$ and $\partial_y F$ from the right and from the left by a linear expressions on $x,y$, or from one side by a quadratic expression. We can not get new relations on monomials of degree 4 this way, again because quadratic parts of  $\partial_x F$ and $\partial_y F$ form a Gr\"obner basis.
%, and thus all ambiguities are resolvable.
However we can get new relations on terms of degree 4 if we multiply $\partial_x F$ and $\partial_y F$ by linear terms from the right or from the left, and terms of degree 3 cancel. Then the relation we get will be of degree 4, so it might create dependance between $x^4$ and $yx^3$.
We want to find out how many of such constrains we can get, to know the dimension of $A_4/A_5$.  For that  we consider two possible ambiguities:
$x^3=x^2\cdot x = x \cdot x^2 $ and $ x^2y= x \cdot xy =x^2 \cdot y $ between $\partial_x F_3= xy+yx$ and $\partial_y F_3=x^2$. The resolution of these two ambiguities  leads to two elements, generating the space $\Omega$ of monomials with zero third component, spanned by $u_1 \cdot \partial_x F_3, \partial_x F_3 \cdot u_2, v_1 \cdot \partial_y F_3, \partial_y F_3 \cdot v_2$,
% \deg u=\deg v=1$,
with  $u_i,v_i$ - linear polynomials.

In this case resolution of ambiguity $x^2y$ will give the relation $[\partial_x F, x]+ [\partial_y F,y]$, which as we know already present in $\Omega$, as it is a syzygy between the relations \ref{syz}. Thus only the ambiguity $x^3$ can bring new relation. In other words, two relations we got from the ambiguities are linearly dependant.
Thus the space $\Omega$ is at most one dimensional and $\dim A_4/A_5$ can be only one or two.

If it is two, then the overall dimension of $A$ we got already is $\sum\limits_{i=0}^{4} \dim A_i/A_{i+1}= 1+2+2+2+2=9$.

If it is one, then we will look at dimension of $A_5/A_6$. It is shown by brute force calculation, that it  is at least one.
Then the overall dimension  we got up to this step is $\dim A= \sum\limits_{i=0}^{5} \dim A_i/A_{i+1}= 1+2+2+2+1+1=9$.
\end{proof}

Let us now formulate a more general fact, which in particular will show that the dimension in this case is $\geq 9$, but also will be used later in classification of degree 9.

First let us note the following

\begin{lemma}\label{zamx2y} Let $A=\kx / id(F)$, where the potential $F=x^2 y\ww+ F_4+...+F_r$, then by composition of invertible changes of variables of the form $x \to x, y\to y+f(x,y)$, or $y \to y, x\to x+g(x,y)$
 where $f(x,y),g(x,y)$ are polynomials of degree $\geq 2$, we can bring the potential to the shape $F=x^2y\ww+y^4p(y)$, for a polynomial $p(y)$.
\end{lemma}

\begin{proof}
Let us list the changes of variables and the corresponding terms of the potential, which could be killed by them:

$$y\to y+cx^2 --- x^4, y\to y+cxy --- x^3y\ww, y \to y+c y^2 --- x^2y^2\ww$$
$$x\to x+cy^2 --- xy^3\ww, x\to x+cxy --- xyxy\ww, y \to y+c y^2 --- x^2y^2\ww.$$

Only monomial $y^4$ will remain in the potential in degree 4.

$$y\to y+cx^3 --- x^5, y\to y+cx^2y --- x^4y\ww, y \to y+c xy^2 --- x^3y^2\ww,$$
$$y\to y+cyxy --- x^2yxy\ww, y\to y+cy^3 --- x^2y^3\ww,$$
$$x\to x+cyx^2 --- x^2yxy\ww, x\to x+cyxy --- xyxy^2\ww, y \to y+c y^3 --- xy^4\ww.$$

Only monomial $y^5$ will remain in the potential in degree 5, etc.

Thus, after composition of such changes of variables we get a potential $F=x^2y\ww+y^4p(y)$, for some polynomial $p(y)$.

\end{proof}

\begin{proposition}\label{gen}
Let $A=\kx / id(F)$, where the potential $F=x^2 y\ww+ p(y)$, $p(y)$ is the polynomial of degree $k$, with even part of degree $2n$. Then the linear basis of $A$ consists of monomials $y^N, xy^{N-1}$ for $n={0,..., 2n+3},$ and $y^N$ for $N = 2n+4,..., 2n+k+5$.
\end{proposition}

\begin{proof} Let us construct a Gr\"obner bases on two relations we have from  the potential:
$$\partial F_x: x^2=y^3 p(y), \partial_y F: xy=-yx.$$

The ambiguity $x^2y $ is resolvable. Indeed, $x^2y \to y^3 p(y)y+xyx\ar-yx^2\ar - y^4p(y)=0$. The ambiguity $x^3$ reduces to $xy^3p(y)-y^3p(y)x=-2y^3 p_{even}(y) x=0$, where $p_{even}(y)$ stands for the even part of the polynomial $p(y)$, since due to the relation  $xy=-yx$, $xy^n=y^nx $, if $n$ is even, and $xy^n=-y^nx$, if $n$ is odd.
Thus from the above we got new relation $$ y^3 p_{even}(y)x=0 .$$

There are no other ambiguities between old relations. Note that in particular this implies that if $p(y)$ is odd, then the algebra is infinite dimensional.

Denote
$$
p_{even}(y)=c y^{2n}+..., n=0,1,..., c\neq 0,
$$
then we can rewrite the latter relation $ y^3 p_{even}(y)x=0 $ as $c y^3y^{2n}(1+u(y))x=0.$
Since $1+u(y)$ is invertible and (anti-)commute with $x$, we have the new relation
$$
y^{2n+3}=0.
$$

Now consider new ambiguities formed by this relation with the old ones.
The ambiguity $xy^{2n+3}x^2$ is resolvable. The ambiguity $$y^{2n+3}x \to y^{2n+6} \to c' y^{2n+6+k}(1+v(y)),$$
where $k$ is degree of $p(y): \,Q p(y)=c' y^k+...$, and $k$ odd. Since $1+v(y)$ is invertible, we have a relation
$$
y^{2n+k+6}=0.
$$
This relation does not produce any unresolvable ambiguities. Thus the Gr\"obner basis consists of
$$
x^2-y^3p(y)=0,\ \  xy+yx,y^{2n+3}=0,\ \   y^{2n+k+6}=0
$$
   and by this the linear basis of algebra is determined. Namely it consists of
    monomials which does not  contain as submonomials leading terms of elements of Gr\"obner basis ($x >y$).

Thus the linear basis of $A$ is:    $y^N, xy^{N-1}$ for $n={0,..., 2n+3},$ and $y^N$ for $N = 2n+4,..., 2n+k+5$.
\end{proof}

We can calculate then the dimension of the algebra, when potential has a cubic term $x^2y\ww.$ Just by counting elements of linear basis described in the proposition, we deduce the following.

\begin{corollary}\label{1}
  If $k=2n,$ $\dim A =3(2n+3)$, if $k<2n, \dim A= 4n+k+9$ ($k$ odd), where $k$ is degree of polynomial $p(x)$, and $2n$ is degree of its even part.
\end{corollary}

This corollary, combined with the Lemma~\ref{zamx2y} provides the minimal possible dimension in case of the potential with cubic term $x^2y\ww$.

\begin{corollary}\label{2} Minimal possible dimension of $A(F)$, where $F=x^2y\ww+F_4+...+F_r$, is 9.
\end{corollary}

\begin{proof} In case $p(y)$ has even degree ($k=2n$), the minimal possible dimension is 9 (corresponds to $n=k=0$). In case $p(y)$ has odd degree
($k <2n$), the minimal possible dimension is 14 (corresponds to $n=k=1$).
\end{proof}

\subsection{Potential with cubic term $x^3$}\label{x3}

Let us consider now potential with cubic term $x^3$. This one provides the fastest growth of the three.

\begin{lemma}\label{x3}
Let $A=\kx/ id_{\kx} (\partial_x F, \partial_y F)$
 for the potential $F \in \k\langle x,y \rangle,
 F=x^3 +F_4+...+F_r.$ Then $\dim A \geq 10$.
\end{lemma}

\begin{proof} As in the previous lemma $A_1/A_2=span <x,y>_{\K}$. Since $\partial_x F =x^2+F_3+F_4+..+F_r, \partial_y F= F_3+F_4+...+F_k$,  the $A_2/A_3 =span<xy, yx,y^2>_{\K} $ consists only of monomials without submonomial $x^2$. Then $A_3/A_4$ spanned by 2-normal monomials $y^3, xy^2, y^2x, xyx, yxy$, the relation $\partial_y F$ which has term  of (lower) degree 3 is the only one which can create  one linear dependance between them $\mod A_4$. Thus $\dim A_3/A_4 \geq 4$, and we get $\dim A \geq 1+2+3+4=10$.
\end{proof}

\subsection{Potential with cubic term $x^3+y^3$}\label{x3y3}

Now we consider the case of potential with $F_3=x^3+y^3$, which can provide algebras of dimension 8.
Our goal will be to prove the following.

\begin{theorem} There is only one up to isomorphism algebra $$A=\kx/ id_{\kx} (\partial_x F, \partial_y F)$$ which has dimension 8. It is the algebra given by the potential
$F=x^3+y^3 +xyxy\ww$.
\end{theorem}

\begin{proof} As before $A_1/A_2=\spann<x,y>_{\K}$. Since $\partial_x F_3=x^2, \partial_y F_3=y^2, A_2/A_3=\spann <xy, yx>_{\K}.$ The 3-normal monomials are $xyx, yxy$. The linear dependances   between them would not occur because $x^2 $ and $y^2$ form a Gr\"obner basis. Thus dimension of $A_2/A_3$ is 2.

Then $A_3/A_4 $ is spanned by 3-monomials $xyx, yxy$. The relation between them can not occur from generating monomials $x^2$ and $y^2$ of the ideal, which trivially form a Gr\"obner basis. Hence $\dim A_3/A_4 =2$.

Then $A_4/A_5$ is spanned by $xyxy, yxyx$. Let us see which relations on degree 4 monomials can we have. The relations   come from resolution of ambiguities $ x^3=x^2 \cdot x=x \cdot x^2$ and $y^3=y \cdot y^2 =y^2 \cdot y$.

These ambiguities produce polynomials $[\partial_x,x]$ and $[\partial_y,y]$. We know that they are linearly dependant, since there is a syzygy $[\partial_x,x]+[\partial_y,y]=0$ \ref{syz}, so we have maximum one relation. These both also can be zero. In the first case $\dim A_4/A_5 =1$, in the second, $\dim A_4/A_5 =2$. If $\dim A_4/A_5 =2$, then, $\dim A_5/A_6 $ is $2$ (if the same ambiguities do not produce nontrivial new relations of degree 5) or 1 (if these ambiguities do produce a relation of degree 5). In both cases whether $\dim A_5/A_6 $ is equal to 2 or 1, we already reached dimension 10 or 11. Thus the only interesting for us here case, which can lead to dimension 8, or potentially 9, is when $\dim A_4/A_5=1$.

We will see that, in fact, it is necessarily lead to $\dim 8$, and $\dim 9$ is impossible. Indeed, if we have one linear dependance between 4-normal monomials $xyxy$ and $yxyx$, then $xyxy=\alpha  yxyx$. If we consider now 5-normal monomials, they should be $xyxyx$ and $yxyxy$, but both of them on the other hand equal to $xyxyx=\alpha yxyxx=0$ and $yxyxy=\alpha yyxyx=0$, since $x^2=y^2=0$. Thus dimension of $A_5/A_6$ must be 0.
%\end{proof}

{\bf Remark}. As a consequence of this argument we see that dimension $ 9$ can not occur from the potential with cubic  part $x^3+y^3$.

Now we need to find out how many non-isomorphic algebras can have a potential of the type $F=x^3+y^3+F^4+...+F_r$.

%\begin{lemma} There is only one up to isomorphism algebra of dimension 8, it is defined by the potential %$F=x^3+y^3+xyxy\ww$.
%\end{lemma}

%\begin{proof}
 To prove that there is only one, we find changes of variables which preserve $F_3=x^3+y^3$ and  kill all terms in the potential of degree four but $xyxy\ww$.
Namely: $$x \to x+\alpha y^2 --- x^2y^2\ww, x \to x+\alpha xy --- x^3y\ww, x \to x+\alpha x^2 --- x^4,$$
$$y\to y+\alpha xy --- xy^3\ww,  y \to \alpha y^2 --- y^4.$$

%of dimension Only term $xyxy\ww$ can not be killed (preserving dimension 8).
Note that these changes of variables are all invertible transformations in $\kx$ (see Proposition~\ref{prop3}), thus we get isomorphic objects.
This proves that there exists only one up to isomorphism algebra of dimension 8, $F_3=x^3+y^3$, it is given by the potential $x^3+y^3+xyxy\ww$.
\end{proof}

%Let us consider first transformation: $x \to x+\alpha y^2, y \to y.$ Let's see where $x^3$ goes under this %transformation:

%$$x^3 \to (x+\alpha y^2) = x^3+ \alpha (x^2y^2+xy^2x+y^2x^2)+ \alpha^2(y^4+y^2xy^2+ y^4x)+\alpha^3 y^6.$$

%...

%\end{proof}

\section{Classification in dimension 9}

As we have seen before the potential with cubic term $x^3+y^3$ can not give dimension 9 (see remark in previous section), and as shown in Section~\ref{x3} in the case of potential with cubic term $x^3$ the dimension is bigger then 10. Thus only the potential with cubic term $xy^2\ww$ can give dimension 9.

By Proposition~\ref{gen} we get dimension 9 only when $n=0$, that is the potential is $P=xy^2+y^4p(y)$, where $p(y)=1+c_1y+...+c_sy^s$.

Since we consider the case $\dim A=9$, we can see that $\dim A_6 =0$: the series of dimensions is $\dim A_0=1$, $\dim A_1=2$, $\dim A_3=2$, $\dim A_4=1$, $\dim A_5=1$:  $1+2+2+2+1+1=9$.

Thus the potential should have the form: $P=xy^2\ww+ y^4+c_2 y^5+c_3 y^6.$

Now notice the following. We can kill the term $y^6$  in the potential by the invertible change of variables:
$x \to x, y \to y+cy^3.$
Thus  the potential can be taken into the form:
$$P=xy^2\ww+ y^4+c y^5.$$

This gives two 9 dimensional algebras, depending on whether $c=0$ or $c\neq 0:$

$$P_1=x^2y\ww+y^4, P_2=x^2y\ww+y^4+cy^5, c\neq 0.$$

Of course, if $c\neq 0$ then by the scaling we can make $c=1$. So, finally the only thing remained to prove the classification of 9 dimensional algebras is the following.

\begin{theorem} Algebras given by the potentials
$$P_1=x^2y\ww+y^4, P_2=x^2y\ww+y^4+y^5$$
are not isomorphic:
$$A_1=\kx/id(\partial_x P_1, \partial_y P_1)\neq A_2=\kx/id(\partial_x P_2, \partial_y P_2).$$
\end{theorem}

\begin{proof} We need to show that algebras $A_1=\kx /id (R_1)$ non-isomorphic to $A_2=\kx /id (R_2)$,
where
$$
R_1: \,\, xy+yx, x^2+y^3,
$$

$$
R_2:\,\, xy+yx, x^2+y^3+y^4.
$$
  Let us  consider corresponding algebras  $A'_1=\k \langle x,y \rangle /id (R_1)$, $A'_2=\k \langle x,y \rangle /id (R_2)$, and calculate their Gr\"obner bases (as an ideal in free algebra) for $x<y$. We get that the Gr\"obner bases of $R_1$ is:
 $$ xy+yx, x^2+y^3, x^3$$
 The first algebra is nilpotent, which implies $A'_1 =\overline {A'_1 }= A_1$.
 When we calculate the  Gr\"obner bases of $R_2$, we see that relation $y^6=y^7p(y)$ holds, thus $A_2=A'_2/(y^6)$.
So, we reduced the problem to isomorphism of two algebras:$A_1=\k \langle x,y \rangle / (R_1)$ and $A_2=\k \langle x,y \rangle / (R_2, y^6).$

Now looking at the Gr\"obner bases of both of them we see that linear bases for them consist of monomials:
$$A_1: 1, x, y, x^2, yx, y^2, yx^2, y^2x, y^2x^2$$

$$A_2: 1, x, y,  yx, y^2, x^2, yx^2, y^2x, y^3.$$

Another possible linear basis for both algebras is:
$$ 1, x, y,   y^2, yx, y^2x,  y^3, y^4, y^5.$$

Multiplication table in this basis for $A_1$ and $A_2$ are as follows.

In
$$A_1: x^2=-y^3-y^4, xy=-yx, xy^2=y^2x, xyx=y^4+y^5, $$
$$xy^2x=-y^5, xy^3=0, xy^4=0, xy^5=0, y^6=0,$$
$$y^3x=0, yx^2=-y^4-y^5, yxy^2x=0, y^2x^2=-y^5...$$

In
$$A_2: x^2=-y^3, xy=-yx, xy^2=y^2x, xyx=y^4, $$
$$xy^2x=-y^5, xy^3=0, xy^4=0, xy^5=0, y^6=0,$$
$$y^3x=0, yx^2=-y^4,  y^2x^2=-y^5...$$

Using this liner basis rewrite relations in a more convenient way:

$$A_1: x^2=-y^3, xy=-yx, xy^2=0$$

$$A_2: x^2=-y^3-y^4, xy=-yx, xy^3=0, x^3=0, y^6=0.$$

Look for isomorphism $\phi: A_1 \to A_2, \phi(x)=\tilde x, \phi(y)=\tilde y $
with undetermined coefficients:
$$\tilde x= a_1x+a_2 y+a_3    yx +a_4 y^2 +a_5 y^2x +a_6  y^3 +a_7 y^4 +a_8 y^5$$

$$\tilde y= b_1x+b_2 y+b_3    yx+b_4 yx^2 +b_5 y^2x +b_6  y^3 +b_7 y^4 +b_8 y^5.$$

Isomorphism means we should have:
$$\tilde x=- \tilde y^3-\tilde y^4, \tilde x \tilde y=-\tilde y \tilde x, \tilde x \tilde y^4=0, \tilde x^3=0, \tilde x^6=0.$$

From this we derive the following conditions on coefficients.
$$\tilde x^3=0 \Longrightarrow \tilde x^3 =a_2^3 y^3+a_1a_2^2y^2x+...=0 \Longrightarrow a_2=0.$$

Now $\tilde x^3=-3a_1^2a_4y^5  \Longrightarrow a_4=0, $ since $\tilde x \notin A_2^2 \Longrightarrow a_1\neq0.$

 Moreover,
 $$\tilde x\tilde y+\tilde y \tilde x=0 \Longrightarrow \tilde x\tilde y+\tilde y \tilde x=-2a_1b_1 y^3+... \Longrightarrow b_1=0.$$
 Note that $\tilde x, \tilde y $ are independent modulo $A_2^2 \Longrightarrow b_2\neq 0.$
We have further
$$ a_1b_4y^2x+... \Longrightarrow b_4=0, 2a_6b_2 y^4+... \Longrightarrow a_6=0.$$

Rewrite $\tilde x, \tilde y$ substituting coefficients we found:
$$\tilde x= a_1x+a_3    yx +a_5 y^2x +a_7 y^4 +a_8 y^5$$

$$\tilde y= b_2 y+b_3    yx +b_5 y^2x +b_6  y^3 +b_7 y^4 +b_8 y^5.$$

We have also $\tilde x^2+\tilde y^3+\tilde y^4=0,$
$$\tilde x^2=a_1^2+a_1a_3(xyx+yx^2)+a_3^2yxyx +a_1a_5(xy^2x+y^2x^2)$$

$$\tilde y^3=b_2^2 y^3 +3 b_2b_6 y^5, \tilde y^4=b^4y^4 \Longrightarrow b_2=0.$$ This is the contradiction with what we noticed before: $b_2 \neq 0.$
Thus the isomorphism does not exist.
\end{proof}

\section{ Associated graded structures of filtered braces  }\label{br}

%Here we will generalise results for the nilpotent braces
%of previous two sections
% to the braces endowed with appropriate descending filtration.
Let $B$ be the brace with filtration $B=\bigcup B_i, \, B_{i+1} \subset B_{i}, \, B_i * B_j \subset B_{i+j}$, we also suppose $\bigcap\limits_{i=1}^{\infty} B_i = 0$. As usual, we say that element $x$  has degree $i$ with respect to this filtration if $x \in B_i, x\notin B_{i-1}$.
%For
Nilpotent braces are naturally endowed with such filtration.
% is defined as it is done in previous section.

We will later consider   also a completion $\widehat B$ of $B$ with respect to a topology, defined by the given filtration.
 %This will serve as a tool for the proof of theorem \ref{gr}.
  We can think of the completion  $\widehat B$ as a set consisting of infinite series $\sum\limits_{i=1}^{\infty} r_i, r_i \in B_i$.
For details of the completion construction  see, for example, \cite{Baht} chapter 8. There is a natural filtration on $\widehat B$: $\widehat B = \bigcup \widehat B_i$, $\widehat B_{i+1} \subset \widehat B_{i}, \, \widehat B_i * \widehat B_j \subset \widehat B_{i+j}$, $\bigcap\limits_{i=1}^{\infty} \widehat B_i = 0$, where $\widehat B_i$ consists of infinite series of degree $i$. We  say here that degree of series $\sum\limits_{i=1}^{\infty}r_i$ is $i$ if $r_i\neq 0, r_m=0$ for $m <i$.

With respect to the filtration on $B$ we can consider associated graded structure $B_{gr}= \oplus B_i/B_{i+1} =\oplus \bar B_i$ with multiplication defined in the following way: for $a_i \in B_i, b_j \in B_j$
$$a_i * b_j=a_i * b_j+B_{i+j+1}$$
Then we extend it to arbitrary (non-homogeneous) elements using left and right distributivity, which holds in $B_{gr}$, because of the nature of the left brace identity. Indeed,

$$(a+b+a*b)*c=a*c+b*c+a*(b*c)$$ for $a\in B_i, b\in B_j, c\in B_k$ means $(a+b)*c=a*c+b*c+B_r, r={\rm max}(i,j)+k+1$.
Thus in $B_{gr}$ we have
right distributivity
$$(a+b)*c=a*c+b*c$$
for $a\in \bar B_i, b\in \bar B_j, c\in \bar B_k$.
The left distributivity  holds in $B$ itself, so trivially in $B_{gr}$ as well.
Now using two-sided distributivity in $B_{gr}$ we can correctly define multiplication in $B_{gr}$ by extending it from homogeneous elements:
for arbitrary $a,b \in B_{gr}$, $a=a_1+...+a_n, b=b_1+...+b_m, a_i\in \bar B_i, b_i \in \bar B_i$
$$(a*b)_k=\sum\limits_{i=1}^{k-1} a_i * b_{k-i}+B_{k+1}.$$

So, this construction correctly defines multiplication since the brace  axiom $(a+b+a*b)*c=a*c+b*c+a*(b*c)$ supplied us with the right distributivity in the associated graded $B_{gr}.$

\begin{proposition}
Let (B,+,*) be a set  with two operations endowed with decreasing filtration:
 $B=\bigcup B_i, \, B_{i+1} \subset B_{i}, \, B_i*B_j \subset B_{i+j}$. Consider associated graded space
$B_{gr}= \oplus B_i/B_{i+1} =\oplus \bar B_i$. If for
$a\in B_i, b\in B_j, c\in B_k$, $(a+b)*c=a*c+b*c+B_r, \,{\rm and}\,\, a*(b+c)=a*b+a*c+B_r, r={\rm max}(i,j)+k+1 $, then
 $B_{gr}$ with multiplication extended from homogeneous components:
 $a,b \in B_{gr}$, $a=a_1+...+a_n, b=b_1+...+b_m, a_i\in \bar B_i, b_i \in \bar B_i$
$$(a*b)_k=\sum\limits_{i=1}^{k-1} a_i * b_{k-i}+B_{k+1}$$
 is correctly defined.
\end{proposition}

We can analogously associate  a graded structure for the completion, which by definition is the direct product of quotients of elements of filtration of $B$ (which are isomorphic to quotients of elements of filtration of $\widehat B$).
$$ \widehat B_{gr}= \prod_{i=1}^{\infty} B_i/B_{i+1} (=\prod_{i=1}^{\infty} \widehat B_i/ \widehat B_{i+1}).$$
Since $\bigcap B_i = 0$ we have $B \subset \widehat B$ is a subbrace of the completion and $B_{gr}$ is obviously  a subbrace of $\widehat B_{gr}.$

\begin{theorem}\label{gr} Let $B$ be a left brace endowed  with descending filtration $B=\bigcup B_i$ where $B_i \triangleleft  B$, $ B_{i+1} \subset B_{i}, \, B_i * B_j \subset B_{i+j}$, satisfying  $\bigcap\limits_{i=1}^{\infty} B_i =0$, then the associated graded
$B_{gr}= \oplus B_i/B_{i+1} =\oplus \bar B_i$ with multiplication defined above is a pre-Lie algebra.
\end{theorem}

For the proof of the theorem we are going first to extend  the right distributivity formula in lemma 15 proved in \cite{Engel} for the nilpotent case to the case of an arbitrary descending ideal filtration with zero intersection. To write down the formula we should extend our realm to the completion  $\widehat B$ of $B$, with respect a given filtration.

\begin{lemma}\label{di} Let $B$ be a left brace endowed with the descending filtration  $B=\bigcup B_i$ where $B_i$ are ideals in $B$, $ B_{i+1} \subset B_{i}, \, B_i * B_j \subset B_{i+j}$, satisfying  $\bigcap\limits_{i=1}^{\infty} B_i =0$. Then the right distributivity formula holds in $\widehat B$:
$$(a+b)*c=a*c+b*c+ \sum\limits_{i=1}^{\infty}(-1)^{i+1} (d_i*d_i')*c -d_i*(d_i'*c).$$
 here $a,b,c \in B$, $d_0=a, d_0'=b$ and for $i \geq 1$ $d_i$th are defined as $d_i=d_{i-1}+d_{i-1}'.$
\end{lemma}

\begin{proof}
This is a direct consequence of the lemma 15 proved in \cite{Engel} in the nilpotent case. Indeed, the equality we need to prove should hold in infinite series, and series coincide means they coincide componentwise (here zero intersection of the filtration is important). For any $n$ we consider  now a nilpotent algebra $B/B_n$ and apply to it the lemma, this is possible since filtration components $B_n$ are ideals in $B$. We get a formula in $B/B_n$, which means the series are coincide up to $n$th term. Since it is true for any $n$ this proves the statement.
\end{proof}

Now we notice one consequence of the formula from previous lemma, which holds in infinite series (completion) corresponding to the filtration.

\begin{corollary} In the setting of previous lemma for arbitrary infinite series  from the completion corresponding to given filtration $a,b,c  \in B, \,\, u_r \in \widehat B$. If $\deg a < \deg b$, then
$$(a+b)*c = a*c+b*c+u_r,$$
where $\deg u_r  > \deg b+ \deg c$.
  \end{corollary}

\begin{proof} This can be seen directly from the formula in Lemma~\ref{di} above.
\end{proof}

\begin{lemma}\label{2} If for $a,b,c \in B, u_r \in \widehat B$ the property
$$(a+b)*c = a*c+b*c+u_r,$$
where $\deg u_r  > \deg b+ \deg c$ holds, then $B_{gr}$ is a pre-Lie algebra.
\end{lemma}

\begin{proof} According to the brace axiom  we have
$$(a+b+a*b)*c = a*c+b*c+a*(b*c),$$
and if we permute $a$ and $b$ we get
$$(b+a+b*a)*c = b*c+a*c+b*(a*c).$$

Now due to the property from the corollary we can present the left hand side as
$$ ((a+b)+a*b)*c=(a+b)*c +(a*b)*c +u_r$$
with $\deg u_r > \deg a*b + \deg c$
and
$$ ((b+a)+b*a)*c=(b+a)*c +(b*a)*c +v_r$$
with $\deg v_r > \deg b*a + \deg c.$
Hence
$$(a+b)*c +(a*b)*c +u_r=a*c+b*c+a*(b*c)$$
and
$$(b+a)*c +(b*a)*c +v_r=b*c+a*c+b*(a*c).$$
Subtracting these two we get
$$(a*b)*c-a*(b*c)=(b*a)c-b*(a*c)+u_r+v_r,$$
hence in $B_{gr}$ we have left-symmetric identity.

\end{proof}

\begin{proof} The combination of the corollary and Lemma~\ref{2} completes the proof of the theorem.
\end{proof}

\section{Associated graded structures to trusses and pre-Lie algebras}\label{tr}

The notion which incorporates the notion of a ring and of a brace was introduced by T. Brzezinski \cite{Brz} and called truss (because the defining law is holding both structures together).

\begin{definition} A set $(A, \circ, +) $ with two binary operations is called a truss, if $(A,+)$ is an abelian group,  $(A,\circ)$ is a semigroup and
$$a\circ(b+c)+a=a\circ b+a\circ c +\alpha(a)$$
where $\alpha$ is some function $\alpha: A \to A.$
\end{definition}

If we rewrite this definition in terms of operation $a *b =A\circ b -a-b$, from the axiom of associativity in $(A,\circ)$ we get the same axiom as in brace:
$$(a*b+a+b)*c=a*c+b*c+a*(b*c)$$
and from the above axiom we get
$$a*(b+c)=a*b+a*c+\alpha(a).$$

We will use axioms of truss in this form.
Let now $B$ be a truss endowed with descending ideal filtration with zero intersection. Under mild condition on $\alpha$, we can consider associated graded to this filtration, and show that it carries a structure of a pre-Lie algebra.

\begin{corollary}\label{Ctr} Let $B$ be a truss endowed with
 descending filtration $B=\bigcup B_i$ where $B_i \triangleleft  B$, $ B_{i+1} \subset B_{i}, \, B_i * B_j \subset B_{i+j}$, satisfying  $\bigcap\limits_{i=1}^{\infty} B_i =0$,  and such that $\deg \alpha(a) > 2$. Then the associated graded structure
$B_{gr}= \oplus B_i/B_{i+1} =\oplus \bar B_i$  is a pre-Lie algebra.
\end{corollary}

\begin{proof}
The same argument as in previous section shows that the associated graded $B_{gr}$ is well defined, since the correcting term $\alpha$ have big enough degree:  $\deg \alpha(a) > 2$. Then based on the first axiom of the truss (in terms of $*$), which do coincide with the axiom of brace, we used in the proofs in previous section, we can  see literally in the same way as for braces, that the associated graded algebra is a pre-Lie algebra.
\end{proof}

%{\bf Acknowledgments.}
% We are very thankful to Wolfgang Rump for  answering our many questions about his construction and for useful comments.
%$ $

%AS acknowledges support from the EPSRC grant EP/R034826/1.

%Add references Zelmanov? Shestakov ? Tatiana?

e-mail address: iyudu@mpim.mpg-de; n.iyudu@ihes.fr; n.joudu@yahoo.de

\end{document}